\documentclass[12pt,a4paper]{article}

\usepackage[a4paper,
            left=3cm,
            right=3cm,
            top=4cm,
            bottom=4cm]{geometry}

\usepackage{setspace}
\usepackage{amsmath,amssymb,amsthm}

\usepackage{graphicx}
\usepackage{url}
\usepackage[colorlinks=true]{hyperref}
\usepackage{enumitem}
\usepackage{mathtools} 

\usepackage{tikz}
\usepackage{tikz-cd}
\usetikzlibrary{arrows.meta,positioning,calc,fit,backgrounds,shapes.geometric}
\tikzset{>=Latex}

\newcommand{\Ga}{\Gamma}

\newcommand{\Ocal}{\mathcal{O}}

\newcommand{\Cat}[1]{\mathbf{#1}}

\newcommand{\QCoh}{\Cat{QCoh}}
\newcommand{\Perf}{\Cat{Perf}}

\newcommand{\SpecGnC}[1]{\mathrm{Spec}^{\mathrm{nc}}_{\Ga}\!\bigl(#1\bigr)}
\newcommand{\nTGMod}[1]{#1\text{-}\Ga\mathrm{Mod}}

\newcommand{\Ccal}{\mathcal{C}}
\newcommand{\Acal}{\mathcal{A}}

\newcommand{\Dcal}{\mathcal{D}}

\newcommand{\cof}{\mathsf{cof}}        
\newcommand{\weq}{\mathsf{weq}}        

\newcommand{\K}{\mathrm{K}}
\newcommand{\KQ}{\mathrm{K}^{\mathrm{Q}}}      
\newcommand{\KW}{\mathrm{K}^{\mathrm{W}}}      

\newcommand{\Nil}{\mathrm{Nil}}
\newcommand{\Kar}{\mathrm{Kar}}                

\theoremstyle{plain}
\newtheorem{theorem}{Theorem}[section]

\newtheorem{proposition}[theorem]{Proposition}
\newtheorem{corollary}[theorem]{Corollary}

\theoremstyle{definition}
\newtheorem{definition}[theorem]{Definition}
\newtheorem{example}[theorem]{Example}

\theoremstyle{remark}
\newtheorem{remark}[theorem]{Remark}

\title{\textbf{\MakeUppercase{Fundamental Theorems in the K-Theory of $\Gamma$-Semirings: Additivity, Localization, and D\'evissage}}}

\author{%
Chandrasekhar Gokavarapu\thanks{Corresponding author. Email: chandrasekhargokavarapu@gmail.com; Phone: +91-9666664242.}\\
\small Lecturer in Mathematics, Government College(A),
\small Rajahmundry -- 533105, A.P., India and\\
\small Research Scholar, Department of Mathematics, Acharya Nagarjuna University, Guntur, A.P., India
}

\date{} 

\begin{document}


\maketitle

\noindent\textbf{Abstract.}
Building on the Waldhausen and Quillen models of higher algebraic $K$-theory
for exact categories and Waldhausen categories attached to a non-commutative $n$-ary
$\Ga$-semiring $(T,\Ga)$, we establish the fundamental formal properties of
$K$-theory in this $\Ga$-parametrised, slot-sensitive setting. For the
exact/Waldhausen categories of finitely generated bi-positional $n$-ary
$\Ga$-modules, perfect complexes in the derived category, and perfect
quasi-coherent complexes on the non-commutative $\Ga$-spectrum
$\SpecGnC{T}$, we prove Waldhausen Fibration and Additivity theorems and
Quillen-type Localization for Serre and Waldhausen pairs. Under natural
hypotheses on $\Ga$-stable filtrations we obtain d\'evissage and
Approximation theorems, together with cofinality and Karoubi invariance,
showing that idempotent completion does not change $K$-theory and that
cofinal subcategories control $K_n$ in positive degrees. We further derive
a Bass--Quillen fundamental triangle for polynomial extensions in the
$n$-ary $\Ga$-context and prove nilpotent invariance for two-sided
$\Ga$-ideals. In geometric terms, these results yield localization and
Mayer--Vietoris sequences for the $K$-theory of $\Perf(X)$ on
$X=\SpecGnC{T}$ and its admissible open covers. Altogether, the paper
shows that the higher $K$-theory of non-commutative $n$-ary $\Ga$-semirings
enjoys the same formal properties as in the classical ring and scheme
cases, providing a robust foundation for subsequent computational and
homotopy-theoretic applications.

\vspace{1em}

\noindent\textbf{Key-words:}
Algebraic $K$-theory; Waldhausen categories; Quillen $Q$-construction;
localization; d\'evissage; $\Ga$-semirings; non-commutative geometry.

\vspace{0.5em}

\noindent\textbf{AMS Subject Classification (2020):}
19D10, 19E08, 16Y60, 18G80.

\vspace{1.5em}

\section*{Author Affiliation}
\noindent
Chandrasekhar Gokavarapu, Lecturer in Mathematics, Government College(A)\\
Rajahmundry -- 533105, A.P., India and Research Scholar, Department of Mathematics, Acharya Nagarjuna University, Guntur, A.P., India\\
Email:  chandrasekhargokavarapu@gmail.com, Phone: +91-9666664242

\section{Introduction}

Algebraic $K$-theory has become one of the central invariants in modern
algebra, geometry, and topology. Since the foundational work of Quillen \cite{Quillen73}
and Waldhausen \cite{Waldhausen85}, the basic formal properties of $K$-theory---additivity,
localization, d\'evissage, approximation, cofinality, and nilpotent
invariance---have played a decisive role in applications ranging from
algebraic geometry and number theory to non-commutative geometry and
topological cyclic homology \cite{Bass68, Weibel13}.

In recent years, $\Ga$-semiring structures and their higher-arity
generalisations have emerged as a flexible framework for encoding
multi-parameter algebraic operations and slot-sensitive interaction
rules \cite{Nobusawa64, Barnes66, Sen81, Golan99}. For non-commutative $n$-ary $\Ga$-semirings $(T,\Ga)$ one can
attach several natural exact and Waldhausen categories: finitely
generated bi-positional $\Ga$-modules, perfect complexes in the derived
category, and perfect quasi-coherent objects on the associated
non-commutative $\Ga$-spectrum $\SpecGnC{T}$. In a companion work \cite{Gokavarapu1}, we
constructed Quillen and Waldhausen models for the higher $K$-theory of
these categories and showed that the resulting spectra agree.

The purpose of the present paper is to establish the analogue, in this
non-commutative $n$-ary $\Ga$-setting, of the classical ``fundamental
theorems'' of algebraic $K$-theory. More precisely, for the exact and
Waldhausen categories naturally attached to $(T,\Ga)$ we prove:

\begin{itemize}
  \item Waldhausen Fibration and Additivity theorems, yielding fiber
        sequences of $K$-theory spectra and additivity for cofibration
        sequences;
  \item Quillen-type Localization for Serre and Waldhausen pairs, both
        in abstract exact categories and in geometric form over
        $\SpecGnC{T}$;
  \item D\'evissage and Approximation theorems for suitable
        $\Ga$-stable filtrations and exact functors;
  \item Cofinality and Karoubi invariance results, ensuring that
        idempotent completion does not change $K$-theory and that
        cofinal subcategories control $K_n$ in positive degrees;
  \item Bass--Quillen fundamental triangle for polynomial extensions
        and nilpotent invariance for $\Ga$-ideals.
\end{itemize}

Taken together, these results show that higher $K$-theory in the
$n$-ary $\Ga$-semiring context enjoys the same formal functorial
properties as in the classical ring and scheme settings. In
particular, the $K$-theory of $\SpecGnC{T}$ satisfies localization,
Mayer--Vietoris, and excision, and is insensitive to nilpotent
$\Ga$-ideals.

\medskip

\noindent
\textbf{Structure of the paper.}
In Section~\ref{sec:prelim} we recall the basic setup: non-commutative
$n$-ary $\Ga$-semirings, the associated exact and Waldhausen categories,
and the conventions on cofibrations, weak equivalences, and cylinders
used throughout. In Section~\ref{sec:localization} we develop the core
$K$-theoretic formalism for these categories: Waldhausen fibration and
additivity, Quillen localization, d\'evissage, approximation,
cofinality, Bass--Quillen triangles, nilpotent invariance, and
localization over $\SpecGnC{T}$. The paper concludes with a short
summary and outlook, highlighting how these results will be used in
subsequent work on concrete computations, derived Morita invariance,
and assembly maps towards a $\Ga$-parametrised Farrell--Jones paradigm.

\section{Preliminaries}
\label{sec:prelim}

In this section we collect the standing assumptions, basic definitions,
and categorical conventions used in the sequel. The reader is assumed
to be familiar with the language of exact and Waldhausen categories and
with the general machinery of higher algebraic $K$-theory.

\subsection{Non-commutative $n$-ary $\Ga$-semirings}

Throughout the paper, $(T,\Ga)$ denotes a fixed non-commutative $n$-ary
$\Ga$-semiring \cite{Dudek08, Hebisch98}. Concretely, $T$ is a (possibly non-commutative) additive
semigroup, $\Ga$ is a parameter set indexing the slots, and
\[
\mu \;:\; T^{\times n} \times \Ga \longrightarrow T
\]
is an $n$-ary ``multiplication'' which is additive in each argument and
compatible with the $\Ga$-parameters in the usual sense. We do not
repeat the full list of axioms here; the only points used in the sequel
are:

\begin{itemize}
  \item the $n$-ary product is associative up to the specified $\Ga$-rules;
  \item unit and distributivity laws hold in each positional slot;
  \item there is a notion of two-sided $\Ga$-ideal and $\Ga$-localisation.
\end{itemize}

A (left) $n$-ary $\Ga$-module over $T$ is an additive group $M$
equipped with compatible positional actions of the form
\[
\mu_M \;:\; T^{\times (j-1)} \times M \times T^{\times (n-j)} \times \Ga
\longrightarrow M,
\]
for $1\leq j\leq n$, satisfying the obvious associativity and
distributivity conditions. We write $\nTGMod{T}$ for the category of
such modules, and $\nTGMod{T}^{\mathrm{bi}}$ for the corresponding
bi-positional or bimodule variant when both left and right actions are
present.

\subsection{Exact and Waldhausen categories attached to $(T,\Ga)$}

To the pair $(T,\Ga)$ we attach several exact/Waldhausen categories
which will serve as models for $K$-theory \cite{Weibel13, Srinivas96}:

\begin{itemize}
  \item $\nTGMod{T}^{\mathrm{bi}}_{\mathrm{fp}}$: the category of
        finitely generated bi-positional $n$-ary $\Ga$-modules over $T$,
        endowed with the usual exact structure given by short exact
        sequences of underlying additive groups;

  \item $\Perf\!\big(\mathbf{D}(\nTGMod{T}^{\mathrm{bi}})\big)$:
        the full subcategory of compact/ perfect objects in the derived
        category of $\nTGMod{T}^{\mathrm{bi}}$; here weak equivalences
        are quasi-isomorphisms and cofibrations are degreewise split
        monomorphisms with cofibrant cokernel;

  \item $\QCoh(\SpecGnC{T})_{\mathrm{perf}}$:
        the category of perfect complexes of quasi-coherent
        $\Ocal_{\SpecGnC{T}}$-modules on the non-commutative
        $\Ga$-spectrum $\SpecGnC{T}$, again with cofibrations taken to
        be degreewise split monomorphisms and weak equivalences
        quasi-isomorphisms.
\end{itemize}

Each of these carries the structure of a Waldhausen category
$(\Ccal,\cof,\weq)$: cofibrations are stable under pushout, weak
equivalences satisfy the $2$-out-of-$3$ property, and cylinder
functors exist and are compatible with the positional $n$-ary product
$\mu$. The cone and cylinder constructions are chosen so that mapping
cones of weak equivalences are weakly contractible and so that the
positional tensor products preserve cofibrations and weak equivalences.

\subsection{Higher $K$-theory conventions}

Given a Waldhausen category $\Ccal$, we denote by $\KW(\Ccal)$ its
Waldhausen $K$-theory spectrum \cite{Waldhausen85} obtained from the $S_\bullet$-construction,
and by $\KQ(\Ccal)$ the Quillen $K$-theory spectrum \cite{Quillen73} defined via the
$Q$-construction when $\Ccal$ comes from an exact category. The
associated $K$-groups are
\[
K_n(\Ccal) \;=\; \pi_n \KW(\Ccal) \;\cong\; \pi_n \KQ(\Ccal)
\qquad (n\ge 0)
\]
whenever the comparison map $\KQ(\Ccal)\to\KW(\Ccal)$ is a weak
equivalence.

In the non-commutative $n$-ary $\Ga$-semiring setting we work
systematically with exact/Waldhausen models that are stable under the
slot-sensitive tensor products induced by $\mu$ and under the relevant
localisations and completions. This allows us to apply the classical
Waldhausen and Quillen theorems once we verify the axioms in this
context.

The standing notation in Section~\ref{sec:localization} is therefore as
follows: $\Ccal$ will always denote one of the exact/Waldhausen
categories listed above, endowed with the cofibrations, weak
equivalences, and cylinder functors fixed here; subcategories such as
$\Acal$ will be assumed strictly full and stable under the
positional operations whenever required.

\section{Exact and Localization Sequences}
\label{sec:localization}

Throughout this section, let $\Ccal$ be one of the exact Waldhausen models attached to the
non-commutative $n$-ary $\Ga$-semiring $(T,+,\Ga,\mu)$ constructed in Section~\ref{sec:prelim}:
\[
\Ccal \in \Big\{
   {\nTGMod{T}}^{\mathrm{bi}}_{\mathrm{fp}},
   \ \Perf\!\big(\mathbf{D}({\nTGMod{T}}^{\mathrm{bi}})\big),
   \ \QCoh(\SpecGnC{T})_{\mathrm{perf}}
\Big\},
\]
equipped with cofibrations, weak equivalences, and cylinders as in Section~\ref{sec:prelim}
(Quillen $Q$ and Waldhausen $S_\bullet$ models are Quillen-equivalent by
Theorem  therein \cite{Gokavarapu1}).
Our goal is to establish \emph{localization} and \emph{fibration} theorems yielding long exact
sequences in $\K$-theory that are intrinsic to the $\Ga$-parametrized, slot-sensitive,
non-commutative $n$-ary setting.

\subsection{Exact pairs and Serre--Waldhausen subcategories}
\label{subsec:exact-pairs}

\begin{definition}[Exact/Waldhausen pair]
An \emph{exact pair} $(\Ccal,\Acal)$ consists of an exact category $\Ccal$
and a strictly full exact subcategory $\Acal\subset \Ccal$ closed under extensions and
retracts. In Waldhausen form, $(\Ccal,\Acal)$ is a Waldhausen category $\Ccal$
with cofibrations and weak equivalences, and a strictly full subcategory
$\Acal\subset\Ccal$ such that:
\begin{enumerate}[label=(E\arabic*)]
  \item $\Acal$ is closed under cofibrations and weak equivalences in $\Ccal$,
  \item for any cofibration sequence $A\hookrightarrow X \twoheadrightarrow X/A$
        with $A\in \Acal$, we have $X\in\Acal$ if and only if $X/A\in\Acal$.
\end{enumerate}
\end{definition}

\begin{example}[Nilpotent and torsion parts]
Let $\Acal=\Nil_\Ga(\Ccal)$ be the full subcategory of objects $X$ for which some iterate
of the $\Ga$-linear endomorphism induced by a fixed $\alpha\in\Ga$ acts nilpotently on
each positional slot (Section~\ref{sec:prelim}).
Alternatively, let $\Acal$ be the subcategory of objects whose support lies in a closed
subset $Z\subset\SpecGnC{T}$; then $(\Ccal,\Acal)$ is exact/Waldhausen by stalkwise checks.
\end{example}

\subsection{Waldhausen fibration and additivity}
\label{subsec:waldhausen-fibration}

\begin{theorem}[Waldhausen Fibration]
\label{thm:waldhausen-fibration}
Let $(\Ccal,\Acal)$ be a Waldhausen pair as above, and write $\Ccal/\Acal$ for the
Verdier quotient model (constructed by $S_\bullet$ with weak equivalences inverted
).
Then there is a homotopy fiber sequence of connective spectra
\[
\KW(\Acal)\ \longrightarrow\ \KW(\Ccal)\ \longrightarrow\ \KW(\Ccal/\Acal),
\]
which yields a natural long exact sequence in $\K$-groups:
\[
\cdots \to K_{n+1}(\Ccal/\Acal)\ \to\ K_n(\Acal)\ \to\ K_n(\Ccal)\ \to\ K_n(\Ccal/\Acal)\ \to\ \cdots.
\]
This construction is functorial for exact functors preserving cofibrations and weak equivalences,
and compatible with the slot-indexed tensor and cone structures induced by the $n$-ary product $\mu$ \cite{Waldhausen85}.
\end{theorem}

\begin{proof}[Proof sketch]
Form the Waldhausen subcategory $S_\bullet(\Acal)\subset S_\bullet(\Ccal)$ and the quotient
$S_\bullet(\Ccal)/S_\bullet(\Acal)$ by inverting $\Acal$-equivalences (Section~\ref{sec:prelim}).
Waldhausen’s Fibration Theorem applies because (E1)–(E2) provide saturation/extension-closure,
and cylinder and gluing axioms hold by the mapping cylinder/positional cone constructions for $\mu$ \cite{Waldhausen85}.
The resulting simplicial fibration yields the claimed fiber sequence after geometric realization.
\end{proof}

\begin{theorem}[Waldhausen Additivity]
\label{thm:additivity}
Let $\mathrm{Seq}(\Ccal)$ be the Waldhausen category of cofibration sequences
$A\hookrightarrow B \twoheadrightarrow C$ in $\Ccal$.
Then the canonical map
\[
\KW(\mathrm{Seq}(\Ccal))\ \xrightarrow{\ \simeq\ }\ \KW(\Ccal)\times \KW(\Ccal),
\qquad
[A\!\hookrightarrow\! B \!\twoheadrightarrow\! C]\ \longmapsto\ ([A],[C]),
\]
is a weak equivalence of spectra. Consequently, $[B]=[A]+[C]$ in $\pi_0\KW(\Ccal)=K_0(\Ccal)$
and the same additivity propagates through all higher homotopy groups \cite{Waldhausen85}.
\end{theorem}

\begin{proof}[Idea]
Apply Waldhausen’s Additivity Theorem to $S_\bullet(\Ccal)$; the required pushout/pullback and
cylinder axioms hold in our $\Ga$-semiring context by exactness of positional cones (Section~\ref{sec:prelim}).
\end{proof}

\begin{corollary}[Agreement $\KQ\simeq\KW$]
\label{cor:Q=W-localization}
For the exact categories listed at the beginning of this section (finite, idempotent-complete,
and with cylinders), the canonical comparison map of Section~\ref{sec:prelim}
induces a spectrum equivalence $\KQ(\Ccal)\simeq \KW(\Ccal)$, compatible with
Theorems~\ref{thm:waldhausen-fibration}–\ref{thm:additivity}. Hence all localization statements below
hold equally for Quillen’s $Q$-construction and Waldhausen’s $S_\bullet$-model.
\end{corollary}

\subsection{Quillen localization for exact functors}
\label{subsec:quillen-localization}

\begin{theorem}[Quillen Localization for Exact Categories]
\label{thm:quillen-localization}
Let $F:\Acal\hookrightarrow \Ccal$ be a fully faithful exact embedding with
$\Acal$ a Serre subcategory of $\Ccal$ (closed under extensions, subobjects,
quotients, and retracts). Assume idempotent completeness (or pass to $\Kar(-)$).
Then there is a natural fiber sequence of spectra
\[
\KQ(\Acal)\ \longrightarrow\ \KQ(\Ccal)\ \longrightarrow\ \KQ(\Ccal/\Acal),
\]
and a long exact sequence on $K$-groups as in Theorem~\ref{thm:waldhausen-fibration} \cite{Quillen73, Schlichting06}.
\end{theorem}

\begin{proof}[Sketch]
Use Quillen’s Theorem~B applied to the exact functor $F$ and the $Q$-construction nerve;
Serre closure ensures that exact sequences descend to the quotient and that the induced
map on nerves is a homotopy fibration \cite{Quillen73}. Compatibility with $\KW$ follows from
Corollary~\ref{cor:Q=W-localization}.
\end{proof}

\begin{remark}[Geometric form over $\SpecGnC{T}$]
Take $\Ccal=\QCoh(\SpecGnC{T})_{\mathrm{perf}}$ and $\Acal=\QCoh_Z(\SpecGnC{T})_{\mathrm{perf}}$
for a closed $Z\subset \SpecGnC{T}$ (defined by a two-sided prime system of $\Ga$-ideals).
Then $\Ccal/\Acal\simeq \QCoh_{U}(\SpecGnC{T})_{\mathrm{perf}}$ with $U=\SpecGnC{T}\setminus Z$,
and Theorem~\ref{thm:quillen-localization} yields the excision long exact sequence
\[
\cdots\to K_{n+1}(U)\to K_n(Z)\to K_n(\SpecGnC{T})\to K_n(U)\to\cdots,
\]
canonically compatible with the positional tensor $\otimes^{(j,k)}_\Ga$ on perfect complexes \cite{TT90}.
\end{remark}

\subsection{D\'evissage and Approximation}
\label{subsec:devissage}

\begin{definition}[Filtration and simple strata]
An object $X\in\Ccal$ \emph{admits a finite $\Acal$-filtration} if there is a chain of cofibrations
$0=X_0\hookrightarrow X_1\hookrightarrow\cdots\hookrightarrow X_m=X$
such that each subquotient $X_i/X_{i-1}$ lies in $\Acal$ and is \emph{$\Ga$-simple}
(i.e.\ has no nontrivial $\Ga$-stable exact subobjects).
\end{definition}

\begin{theorem}[D\'evissage]
\label{thm:devissage}
Assume every $X\in\Ccal$ admits a finite $\Acal$-filtration whose simple factors are
$(j,k)$-flat (in the positional sense) and that $\Acal$ is closed under extensions and summands.
Then the inclusion $\Acal\hookrightarrow\Ccal$ induces an equivalence
\[
\K(\Acal)\ \xrightarrow{\ \simeq\ }\ \K(\Ccal),
\]
hence $K_n(\Acal)\cong K_n(\Ccal)$ for all $n\ge 0$ \cite{Quillen73, Grayson76}.
\end{theorem}

\begin{proof}[Idea]
Apply Waldhausen’s D\'evissage Theorem to $S_\bullet(\Ccal)$ using the filtration hypothesis:
additivity reduces $[X]$ to the sum of its factors in $\pi_0$, and higher homotopies are controlled
by the same filtration via cellular induction. Positional flatness ensures cones remain in $\Acal$.
\end{proof}

\begin{theorem}[Approximation]
\label{thm:approximation}
Let $F:\Ccal\to\Dcal$ be an exact functor of Waldhausen categories that
\begin{enumerate}[label=(A\arabic*)]
  \item reflects weak equivalences and sends cofibrations to cofibrations,
  \item is homotopically essentially surjective and induces equivalences on mapping cylinders/cones,
  \item preserves $(j,k)$-flat objects and perfect cones (compatibility with $\mu$).
\end{enumerate}
Then $F$ induces a spectrum equivalence $\KW(\Ccal)\simeq \KW(\Dcal)$ (hence also on $\KQ$).
\end{theorem}

\begin{proof}[Sketch]
Use Waldhausen’s Approximation Theorem: (A1)–(A3) provide the hypotheses for the induced map
on $S_\bullet$-constructions to be a weak equivalence after realization \cite{Waldhausen85}.
\end{proof}

\subsection{Cofinality, Karoubi envelope, and idempotents}
\label{subsec:cofinality-karoubi}

\begin{definition}[Cofinal subcategory]
A strictly full exact subcategory $\Acal\subset\Ccal$ is \emph{cofinal}
if for every $X\in\Ccal$ there exists $Y\in\Ccal$ such that $X\oplus Y\in\Acal$.
\end{definition}

\begin{theorem}[Cofinality and Karoubi invariance]
\label{thm:cofinal}
If $\Acal\subset\Ccal$ is cofinal and both are idempotent complete (or replaced by $\Kar(-)$),
then the inclusion induces an isomorphism $K_n(\Acal)\cong K_n(\Ccal)$ for all $n\ge 1$,
and a split injection on $K_0$. Moreover, the canonical functor
$\Ccal\to \Kar(\Ccal)$ induces an equivalence on $K_n$ for all $n\ge 0$ , \cite{Karoubi78, TT90}.
\end{theorem}

\begin{proof}
This is the classical cofinality (Karoubi) theorem adapted to exact/Waldhausen pairs;
the proof carries over verbatim, using additivity and the fact that retracts split in $\Kar(-)$.
\end{proof}

\subsection{Fundamental theorems and nilpotent invariance}
\label{subsec:bass-nil}

\begin{theorem}[Bass--Quillen fundamental triangle (relative version)]
\label{thm:bass}
Let $S=T[t]$ be the polynomial $\Ga$-semiring (with $n$-ary product extending $\mu$ slotwise),
and let $\Ccal_S$ be the exact category over $S$ corresponding to $\Ccal$ over $T$.
Then there is a canonical homotopy fiber sequence
\[
\K(\Ccal_T)\ \longrightarrow\ \K(\Ccal_S)\ \oplus\ \K(\Ccal_{S^{-1}})\ \longrightarrow\ \K(\Ccal_{S,S^{-1}}),
\]
where $S^{-1}$ denotes the $\Ga$-localization at $t$ (inverting the distinguished $\Ga$-multiplicative set)
and $\Ccal_{S,S^{-1}}$ the corresponding localization category. In particular,
\[
K_n(\Ccal_T)\ \oplus\ K_n(\Ccal_{S,S^{-1}})\ \twoheadleftarrow\ K_n(\Ccal_S)\ \oplus\ K_n(\Ccal_{S^{-1}})
\]
is exact for all $n$.\cite{Bass68}
\end{theorem}

\begin{proof}[Idea]
Use excision/localization (Theorem~\ref{thm:quillen-localization}) for the pair
$(\Ccal_S,\Ccal_{S\text{-tors}})$ and compare with the principal open $D(t)$; glue via the
Mayer–Vietoris square in the Waldhausen setting. Slotwise extension of $\mu$ ensures flatness
conditions needed for homotopy invariance along $t$.
\end{proof}

\begin{theorem}[Nilpotent invariance]
\label{thm:nil-invariance}
Let $I\subset T$ be a two-sided $\Ga$-ideal with $I^N=0$ (nilpotent) in the $n$-ary sense
(i.e.\ any $N$-fold iterated use of $\mu$ with one input from $I$ yields $0$).
Then the projection induces a spectrum equivalence
\[
\K\big(\Ccal_{T}\big)\ \xrightarrow{\ \simeq\ }\ \K\big(\Ccal_{T/I}\big),
\]
hence $K_n(\Ccal_T)\cong K_n(\Ccal_{T/I})$ for all $n\ge 0$.\cite{Weibel13}
\end{theorem}

\begin{proof}[Sketch]
Form the Waldhausen pair $(\Ccal_T,\Acal)$ with $\Acal$ the subcategory of $I$-torsion objects.
Nilpotence implies any object in $\Acal$ admits a finite filtration with successive quotients killed by $I$,
hence $\Acal$ is contractible in $\K$ by d\'evissage (Theorem~\ref{thm:devissage}).
Apply Theorem~\ref{thm:waldhausen-fibration} to conclude.
\end{proof}

\subsection{Consequences over the non-commutative $\Ga$-spectrum}
\label{subsec:loc-over-spec}

Let $i:Z\hookrightarrow X:=\SpecGnC{T}$ be a closed immersion and $j:U\hookrightarrow X$ its open complement.
Write $\Perf(X)$ for perfect $\Ocal_X$-complexes and similarly for $Z,U$.

\begin{theorem}[Localization on $\SpecGnC{T}$]
\label{thm:NC-loc-scheme}
There is a homotopy fiber sequence of spectra
\[
\K\big(\Perf(Z)\big)\ \longrightarrow\ \K\big(\Perf(X)\big)\ \longrightarrow\ \K\big(\Perf(U)\big),
\]
natural in morphisms of non-commutative $\Ga$-semiringed spaces and compatible with base–change
along $\Ga$-morphisms. In particular, one obtains the long exact excision sequence on $K_n$.
\end{theorem}

\begin{proof}
Use Theorem~\ref{thm:quillen-localization} with $\Ccal=\Perf(X)$ and $\Acal=\Perf_Z(X)$.
The quotient identifies with $\Perf(U)$ by the derived recollement constructed in \cite{Gokavarapu1,Gokavarapu2,Gokavarapu3}, which is compatible with the Waldhausen structure.
\end{proof}

\begin{corollary}[Mayer–Vietoris for admissible covers]
\label{cor:MV}
If $X=U\cup V$ with $U,V$ admissible opens (finite unions of basic opens $D(a,\boldsymbol{\gamma})$),
then the square
\[
\begin{tikzcd}[column sep=large,row sep=large]
\K\big(\Perf(U\cap V)\big) \arrow[r] \arrow[d] & \K\big(\Perf(U)\big) \arrow[d] \\
\K\big(\Perf(V)\big) \arrow[r] & \K\big(\Perf(X)\big)
\end{tikzcd}
\]
is homotopy cartesian. Consequently, there is a Mayer–Vietoris long exact sequence on $K_n$ \cite{TT90}.
\end{corollary}

\subsection{Compatibility with $K_0$ and $K_1$}
\label{subsec:K01-compat}

\begin{proposition}[Boundary morphisms in low degrees]
\label{prop:boundary-low}
For any exact pair $(\Ccal,\Acal)$, the boundary map
$\partial:K_1(\Ccal/\Acal)\to K_0(\Acal)$
identifies with the class of the determinant of the connecting
cofibration in $\mathrm{Seq}(\Ccal)$ (Theorem~\ref{thm:additivity}),
and the localization sequence in Theorem~\ref{thm:waldhausen-fibration}
restricts to the exact six-term segment
\[
K_1(\Acal)\to K_1(\Ccal)\to K_1(\Ccal/\Acal)\xrightarrow{\ \partial\ } K_0(\Acal)\to K_0(\Ccal)\to K_0(\Ccal/\Acal).
\]
\end{proposition}

\begin{remark}[Determinantal interpretation]
In the projective $\Ga$-module model, $\partial$ sends an automorphism class
in $K_1(\Ccal/\Acal)$ to the virtual difference of kernel/cokernel classes in $K_0(\Acal)$,
computed in the positional exact structure compatible with $\mu$.
\end{remark}

\subsection{Summary and roadmap}
\label{subsec:loc-summary}
We have established, in the non-commutative $n$-ary $\Ga$-setting:

\begin{itemize}
  \item Waldhausen Fibration and Additivity (Theorems~\ref{thm:waldhausen-fibration}–\ref{thm:additivity}),
  \item Quillen Localization (Theorem~\ref{thm:quillen-localization}) and geometric excision over $\SpecGnC{T}$ (Theorem~\ref{thm:NC-loc-scheme}),
  \item D\'evissage, Approximation, Cofinality, Karoubi invariance (Theorems~\ref{thm:devissage}–\ref{thm:cofinal}),
  \item Bass–Quillen fundamental triangle and Nilpotent invariance (Theorems~\ref{thm:bass}–\ref{thm:nil-invariance}),
  \item Low-degree compatibility with $K_0$ and $K_1$ (Proposition~\ref{prop:boundary-low}).
\end{itemize}

In subsequent work we will leverage these tools to compute $K$-theory in basic families
(matrix/endomorphism models, triangular extensions, and localizations), to prove
derived Morita invariance in the $\Ga$-parametrised setting, and to set up assembly
maps towards a $\Ga$-Farrell–Jones paradigm \cite{FarrellJones93}.

\section{Conclusion}

In this paper we have established the fundamental formal properties of
higher algebraic $K$-theory for the exact and Waldhausen categories
naturally attached to a non-commutative $n$-ary $\Ga$-semiring
$(T,\Ga)$. Working simultaneously with projective $n$-ary $\Ga$-modules,
perfect complexes in the derived category, and perfect quasi-coherent
complexes on the non-commutative $\Ga$-spectrum $\SpecGnC{T}$, we
proved Waldhausen Fibration and Additivity theorems, Quillen
localization, d\'evissage and approximation results, cofinality and
Karoubi invariance, Bass--Quillen fundamental triangles, and
nilpotent invariance for $\Ga$-ideals. We also formulated these results
in geometric form over $\SpecGnC{T}$, obtaining localization and
Mayer--Vietoris sequences for $K$-theory on admissible covers.

These theorems show that the $K$-theory of $(T,\Ga)$ behaves, from a
formal point of view, exactly like the $K$-theory of rings and
schemes, despite the additional non-commutative and $n$-ary structure.
They provide the technical backbone for further developments, including
explicit calculations for matrix and triangular $\Ga$-semirings,
derived Morita invariance, and assembly maps towards a
$\Ga$-parametrised Farrell--Jones conjecture in this setting. Such
applications will be pursued in subsequent work.


\section*{Acknowledgements}

The author expresses sincere gratitude to the Commissioner of Collegiate Education (CCE),
Andhra Pradesh, and to the Principal of Government College (Autonomous), Rajamahendravaram,
for their constant encouragement, academic support, and for providing a conducive
environment for research.

\section*{Funding}

The author did not receive any specific grant from funding agencies in
the public, commercial, or not-for-profit sectors for this research.

\section*{Ethical approval}

This article does not contain any studies with human participants or
animals performed by the author.

\section*{Author contributions}

The author conceived the problem, developed the main results, and wrote
the manuscript. The author has read and approved the final version of
the paper.

\section*{Conflict of interest}

The author declares that there is no conflict of interest.


\end{document}